\documentclass[12pt, reqno]{amsart}
\usepackage{amsmath, amsthm, amscd, amsfonts, amssymb, graphicx, color}
\usepackage[bookmarksnumbered, colorlinks, plainpages]{hyperref}
\hypersetup{colorlinks=true,linkcolor=red, anchorcolor=green, citecolor=cyan, urlcolor=red, filecolor=magenta, pdftoolbar=true}
\usepackage{mathtools}
\usepackage{mathrsfs}
\usepackage{enumerate}
\usepackage{verbatim}
\usepackage{cleveref}
\usepackage[utf8]{inputenc}
\usepackage[english]{babel}

\newtheorem{theorem}{Theorem}[section]
\newtheorem{lemma}[theorem]{Lemma}

\theoremstyle{definition}
\newtheorem{definition}[theorem]{Definition}
\newtheorem{example}[theorem]{Example}
\newtheorem{remark}[theorem]{Remark}
\numberwithin{equation}{section}

\newcommand{\R}{\mathbb R} 
\newcommand{\N}{\mathbb N} 
\newcommand{\Q}{\mathbb Q} 
\newcommand{\T}{\mathbb T} 
\newcommand{\Meas}{\mathcal{M}} 
\newcommand{\Prob}{\mathcal{P}} 
\newcommand{\abs}[1]{\left\lvert#1\right\rvert} 
\newcommand{\norm}[1]{\left\lVert#1\right\rVert} 
\newcommand{\iprod}[1]{\left\langle#1\right\rangle} 

\DeclareMathOperator*{\esssup}{ess\,sup}
\DeclareMathOperator{\SGN}{sign}
\DeclareMathOperator{\EX}{\mathbf{E}}
\DeclareMathOperator{\PR}{\mathbf{P}}
\DeclareMathOperator{\CHT}{\mathbf{1}}

\begin{document}

\setcounter{page}{1}

\title[Metric functionals on $L_{p}$]{Characterizing the metric compactification of $L_{p}$ spaces by random measures}

\author[A. W. Guti{é}rrez]{Armando W. Guti{é}rrez}

\address{Department of Mathematics and Systems Analysis, Aalto University, Otakaari 1 Espoo, Finland}

\email{\textcolor[rgb]{0.00,0.00,0.84}{wladimir.gutierrez@aalto.fi}}

\subjclass[2010]{Primary 54D35; Secondary 46E30, 46B20, 60G57}

\keywords{metric compactification, horofunction compactification, metric functional, random measure, Banach spaces}


\begin{abstract}
	We present a complete characterization of the metric compactification of $L_{p}$ spaces
	for $1\leq p < \infty$. Each element of the metric compactification of $L_{p}$
	is represented by a random measure on a certain Polish space. By way of illustration, 
	we revisit the $L_{p}$--mean ergodic theorem for $1 < p < \infty$, and
	Alspach's example of an isometry on a weakly compact convex subset of $L_{1}$
	with no fixed points.
\end{abstract} 
\maketitle
\section{Introduction}
The metric compactification, which is known as the horofunction compactification in the 
setting of proper geodesic metric spaces, has been extensively studied during the last 
twenty years. Remarkably, the foundations of this compactification can be traced back 
to the early 1980s when Gromov \cite{Gromov1981, Ballmann_Gromov_Schroeder1985}
introduced a technique to attach certain boundary points at infinity of every proper 
geodesic metric space. In 2002, Rieffel \cite{Rieffel2002} identified the metric 
compactification of every complete locally compact metric space with the maximal 
ideal space of a unital commutative $C^{*}$-algebra. 

The modern procedure for constructing 
the metric compactification of general (not necessarily proper) metric spaces was discussed 
in \cite{Gaubert_Vigeral2012, Maher_Tiozzo2016}. Every metric space is continuously injected 
into its metric compactification which becomes metrizable provided that the metric space is 
separable. Moreover, surjective isometries between metric spaces can be extended continuously 
to homeomorphisms between the corresponding metric compactifications.

Banach spaces form an important class of metric spaces for which the metric compactification
deserves further study. Complete characterizations of the metric compactification of 
finite- and infinite-dimensional $\ell_{p}$ spaces were presented by the author in 
\cite{Gutierrez2017} and \cite{Gutierrez2018}, respectively. More studies in this context 
can be found in \cite{Karlsson_Metz_Noskov2006, Walsh2007, Lizhen_Schilling2017, Walsh2018}.

The purpose of this paper is to give a complete characterization of the metric compactification 
of the Banach space $L_{p}(\Omega,\Sigma,\PR)$ for all $1\leq p < \infty$, where 
$(\Omega,\Sigma,\PR)$ is a non-atomic standard probability space. 

Several works have confirmed the importance of the metric compactification as a topological 
and geometric tool for the study of isometry groups
\cite{Klein_Nicas2009, Lemmens_Walsh2011, Walsh2011ACTNIL, Walsh2014, Walsh2014HG},
random walks on hyperbolic groups \cite{Bjorklund2010, Gouezel2017},
random product of semicontractions 
\cite{Karlsson_Margulis1999, Karlsson_Ledrappier2011, Gouezel_Karlsson2015, Gouezel2018ICM},
Denjoy-Wolff theorems 
\cite{Karlsson2001, Karlsson2014, Budzynska_Kuczumow_Reich2013, Abate_Raissy2014, Lemmens_Lins_Nussbaum_Wortel2018},
Teichm\"uller spaces \cite{Alessandrini_Liu_Papadopoulos_Su2016}, 
limit graphs \cite{Dangeli_Donno2016}, random triangulations \cite{Curien_Menard2018},
and first-passage percolation \cite{Auffinger_Damron_Hanson2015}. This list is by no means 
exhaustive, but it gives the reader a brief overview of the variety of applications where 
the metric compactification or some of its elements appear. The notion of the metric 
compactification plays an essential role in the development of a metric spectral theory 
proposed by Karlsson \cite{Karlsson2018}. 

\section{Preliminaries}
Let $(X,d)$ be a metric space with an arbitrary base point $x_{0}\in X$. Consider the mapping 
$y\mapsto\mathbf{h}_{y}$ from $X$ into $\R^{X}$ defined by 
\begin{equation}\label{GenMetricFunc}
	\mathbf{h}_{y}(\cdot) := d(\cdot,y)-d(x_{0},y).
\end{equation}
Denote by $\mathrm{Lip}_{x_{0}}^{1}(X;\R)$ the space of $1$-Lipschitz real-valued functions 
on $X$ vanishing at $x_{0}$. It is straightforward to verify that 
\begin{equation}\label{inclusions}
	\{ \mathbf{h}_{y} \,\vert\, y\in X \} \,
		\subset \, \mathrm{Lip}_{x_{0}}^{1}(X;\R) \,
		\subset \prod_{x\in X}[-d(x_{0},x),d(x_{0},x)] \,
		\subset \R^{X}.
\end{equation} 
Tychonoff's theorem asserts that the Cartesian product in \eqref{inclusions} is compact in
the topology of pointwise convergence, which in turn implies that the space 
$\mathrm{Lip}_{x_{0}}^{1}(X;\R)$ is also compact as it is closed in this topology. 

\begin{definition}
The \emph{metric compactification} of $(X,d)$, denoted by $\overline{X}^{h}$, is defined to be 
the pointwise closure of the family $\{ \mathbf{h}_{y} \,\vert\, y\in X \}$. Every element 
$\mathbf{h}\in\overline{X}^{h}$ is called a \emph{metric functional} on $X$. Metric functionals 
of the form \eqref{GenMetricFunc} are called \emph{internal}.
\end{definition}

For every metric functional $\mathbf{h}\in\overline{X}^{h}$, there exists a net of points 
$(y_{\alpha})_{\alpha}$ in $X$ such that the net of internal metric functionals
$(\mathbf{h}_{y_{\alpha}})_{\alpha}$ converges pointwise on $X$ to $\mathbf{h}$. If the metric 
space is separable, the topology in $\overline{X}^{h}$ is metrizable. Hence, sequences are 
sufficient to describe metric funcionals on separable metric spaces. Also, the choice of the 
basepoint $x_{0}$ is irrelevant as different basepoints produce homeomorphic metric 
compactifications. See \cite{Gaubert_Vigeral2012, Maher_Tiozzo2016} for more details.

In general, the mapping $y\mapsto \mathbf{h}_{y}$ given by \eqref{GenMetricFunc} is always 
a continuous injection. It becomes a homeomorphism onto its image whenever the metric space 
$(X,d)$ is proper\footnote{Every closed and bounded subset of $X$ is compact.} and 
geodesic\footnote{For every pair of points $x,y\in X$, there is an isometry from the interval 
$[0,d(x,y)]$ into $X$.}. In this particular case, the metric compactification coincides
with Gromov's original definition of horofunction compactification which is obtained as the 
closure of $\{ \mathbf{h}_{y} \,\vert\, y\in X \}$ with respect to the topology of uniform 
convergence on bounded subsets of $X$; see \cite{Karlsson2018}.

\begin{example}
The set $\R$ of real numbers with the metric induced by the absolute value is proper and 
geodesic. Its metric/horofunction compactification $\overline{\R}^{h}$ contains the internal
metric functionals 
\begin{equation}\label{Rinternalmet}
	s\mapsto \abs{s-r}-\abs{r} \text{ with } r\in\R,
\end{equation}
and the two additional "points at infinity" given by
\begin{equation}\label{Rinfinitemet}
	s\mapsto -s, \quad s\mapsto s. 
\end{equation}
\end{example}

The elements of the compact space $\overline{\R}^{h}$ are used throughout the following sections, 
so it is convenient to establish the following special notations: functions of the form \eqref{Rinternalmet} are 
denoted by $\eta_{r}$ with $r\in\R$, and the two functions in \eqref{Rinfinitemet} are denoted
by $\eta_{+\infty}$ and $\eta_{-\infty}$, respectively. Therefore, the metric compactification 
of $(\R,\abs{\cdot})$ is the compact Polish space 
\begin{equation}\label{Rcompactified}
	\overline{\R}^{h}=\{\eta_{r}\,\vert\, r\in\R\}\cup\{\eta_{+\infty},\eta_{-\infty}\}.
\end{equation} 

Throughout the paper, we assume that $(\Omega,\Sigma,\PR)$ is a non-atomic standard probability 
space. We adopt the convention that all equalities involving measurable sets or measurable 
functions are assumed to hold modulo $\PR$-null sets. For every $1\leq p <\infty$, we denote by 
$L_{p}(\Omega,\Sigma,\PR)$, or simply $L_{p}$ when no confusion arises, the Banach space of 
measurable functions $f:\Omega\to\R$ with finite $L_{p}$-norm
\begin{equation*}
	\norm{f}_{L_{p}}:=\big( \EX[\abs{f}^{p}] \big)^{1/p} =
		\bigg(\int_{\Omega}\abs{f(\omega)}^{p}d\PR(\omega)\bigg)^{1/p}.
\end{equation*}
For simplicity we choose the zero function as the basepoint. For each $g\in L_{p}$, the internal 
metric functional $\mathbf{h}_{g}$ in \eqref{GenMetricFunc} becomes  
\begin{equation}\label{MetricFuncLp}
	\mathbf{h}_{g}(\cdot) = \norm{\cdot-g}_{L_{p}}-\norm{g}_{L_{p}}.
\end{equation} 
The metric compactification $\overline{L_{p}}^{h}$ of the Banach space $L_{p}$ is the set of all 
pointwise accumulation points of \eqref{MetricFuncLp}. We present explicit formulas for these limits 
by means of random measures on $\overline{\R}^{h}$. 

The notion of random measure was introduced 
by Aldous \cite{Aldous1981}; however, there are various equivalent approaches to random
measures \cite{Garling1982, Crauel2002, Kallenberg2017}.
Let $\Meas(\overline{\R}^{h})$ denote the space of all signed Borel measures $\mu$ on the compact
space $\overline{\R}^{h}$ such that the total variation $\abs{\mu}(\overline{\R}^{h})$ is finite. 
The linear space $\Meas(\overline{\R}^{h})$ becomes a Banach space with respect to the norm 
$\norm{\mu}_{\Meas(\overline{\R}^{h})} = \abs{\mu}(\overline{\R}^{h})$. 
Let $C(\overline{\R}^{h})$ denote the space of all real-valued continuous functions on the compact 
space $\overline{\R}^{h}$. The linear space $C(\overline{\R}^{h})$ endowed with the norm 
\begin{equation*}
	\norm{\varphi}_{C(\overline{\R}^{h})}=\sup_{\eta\in\overline{\R}^{h}}\abs{\varphi(\eta)}
\end{equation*}
becomes a Banach space.
The Riesz representation theorem (see e.g., \cite[p.~78]{Albiac_Kalton2016}) asserts that the 
dual space $C(\overline{\R}^{h})^{*}$, equipped with the usual dual norm, 
can be identified with $\Meas(\overline{\R}^{h})$ under 
the isometric isomorphism $\iota :\Meas(\overline{\R}^{h}) \to C(\overline{\R}^{h})^{*}$,
$\mu\mapsto\iota(\mu)$ given by
\begin{equation*}
	\iprod{\varphi,\iota(\mu)} =
		\int_{\overline{\R}^{h}}\varphi(\eta)d\mu(\eta) \;\text{ for all } \varphi \in C(\overline{\R}^{h}). 
\end{equation*}
We denote by $\Prob(\overline{\R}^{h})$ the set of all Borel probability measures on $\overline{\R}^{h}$,
i.e.,
\begin{equation*}
	\Prob(\overline{\R}^{h}) 
		= \Big\{ \mu \in \Meas(\overline{\R}^{h})\ \Big|
			\begin{array}{l}
			\iprod{\varphi,\mu} \geq 0 \quad \forall \varphi \in C(\overline{\R}^{h}), \varphi \geq 0 \\
			\mu(\overline{\R}^{h}) = 1
		  	\end{array}\Big\}.	
\end{equation*}
Let $\Lambda_{\infty}(\Omega,\Meas(\overline{\R}^{h}))$ denote the space of mappings 
$\omega\mapsto\xi_{\omega}$ from $\Omega$ to $\Meas(\overline{\R}^{h})$ with the following
properties
\begin{enumerate}[(i)]
	\item the function $\omega\mapsto \iprod{\varphi,\xi_{\omega}}$ is measurable 
	for all $\varphi\in C(\overline{\R}^{h})$,
	\item the function $\omega\mapsto \norm{\xi_{\omega}}_{\Meas(\overline{\R}^{h})}$
	is essentially bounded.
\end{enumerate}
We say that an element $\xi$ of $\Lambda_{\infty}(\Omega,\Meas(\overline{\R}^{h}))$	is a 
\emph{random measure} on $\overline{\R}^{h}$ whenever $\xi_{\omega}\in\Prob(\overline{\R}^{h})$ 
for $\PR$-almost every $\omega\in\Omega$. Additionally, if 
$\xi_{\omega}\left(\{\eta_{-\infty},\eta_{+\infty}\}\right)=0$ for $\PR$-almost every $\omega\in\Omega$, 
we say that $\xi$ is a random measure on $\R$.

\section{Main results}
\begin{theorem}\label{L1compactification}	
	If $\mathbf{h}\in \overline{L_{1}}^{h}$ then there exists a random measure 
	$\xi$ on $\overline{\R}^{h}$ such that
	\begin{equation}\label{MetricFuncL1}
		 \mathbf{h}(f)=\EX\left[\int_{\overline{\R}^{h}}\eta(f)d\xi(\eta)\right],
	\end{equation}
	for all $f\in L_{1}$. Conversely, if $\xi$ is a random measure on $\overline{\R}^{h}$ 
	then \eqref{MetricFuncL1} defines an element of $\overline{L_{1}}^{h}$.
\end{theorem}

\begin{remark}
	Since $\xi$ and $f$ are defined on $\Omega$, the expectation operator in the representation 
	formula \eqref{MetricFuncL1} should be interpreted as 
	\begin{equation*}
			\EX\left[\int_{\overline{\R}^{h}}\eta(f)d\xi(\eta)\right] =
			\int_{\Omega}\int_{\overline{\R}^{h}}\eta(f(\omega))d\xi_{\omega}(\eta)d\PR(\omega).
	\end{equation*}	 
\end{remark}

\begin{example}
Let $A$ and $B$ be two disjoint measurable subsets of $\Omega$. Suppose that $\xi$ is a random 
measure on $\overline{\R}^{h}$ defined by
\begin{align*}
	\xi_{\omega}:=\begin{cases}
						\delta_{\eta_{+\infty}}, \text{ if } \omega\in A,\\
						\delta_{\eta_{-\infty}}, \text{ if } \omega\in B,\\
						\delta_{\eta_{g(\omega)}}, \text{ otherwise}, 					
					\end{cases}
\end{align*}
where $g:\Omega\to\R$ is a measurable function. Then the metric functional on $L_{1}$ given by 
the formula \eqref{MetricFuncL1} becomes
\begin{equation}\label{sMFL1}
	\mathbf{h}(f)=
		\EX\left[-\CHT_{A}f\right] + \EX\left[\CHT_{B}f\right] +
		\EX\left[\CHT_{\Omega\setminus(A\cup B)}(\abs{f-g}-\abs{g})\right], 
\end{equation}	
for all $f\in L_{1}$. The metric functionals of the form \eqref{sMFL1} can be compared with 
those describing the metric compactification of the sequence space $\ell_{1}$; 
see \cite{Gutierrez2018}. Furthermore, if $A$ and $B$ are sets of measure zero, and $g\in L_{1}$ 
then \eqref{sMFL1} becomes the internal metric functional 
$\mathbf{h}_{g}(\cdot)=\norm{\cdot - g}_{L_{1}}-\norm{g}_{L_{1}}$.
\end{example}

\begin{example}
Let $\Omega$ be the open interval $\left[0,1\right]$ and let $\PR$ be the Lebesgue measure. 
For each $n\in\N$, let $A_{n}$ and $B_{n}$ be the intervals
\begin{equation*}
	A_{n} = \left[\frac{1}{2}-\frac{1}{n+1}\;,\;\frac{1}{2}\right], \quad
	B_{n} = \left]\frac{1}{2}\;,\;\frac{1}{2}+\frac{1}{n+1}\right],
\end{equation*}
and define $g_{n}\in L_{1}$ by $g_{n}:= -n^{2}\CHT_{A_{n}} + n^{2}\CHT_{B_{n}}$.
It follows that
\begin{equation*}
	\norm{g_{n}}_{L_{1}}=2n^{2}/(n+1) \to \infty \text{ as } n \to \infty. 
\end{equation*}
In particular, the sequence
$(g_{n})_{n\in\N}$ does not converge to zero in $L_{1}$-norm. However, for every $f\in L_{1}$ 
we obtain
\begin{align*}
	\mathbf{h}_{g_{n}}(f) 
		&= \norm{f - g_{n}}_{L_{1}}-\norm{g_{n}}_{L_{1}} \\
		&= \EX\left[\abs{f-g_{n}}-\abs{g_{n}}\right] \\
		&= \EX\left[\CHT_{A_{n}} \left(\abs{f+n^{2}}-n^{2}\right)\right]
			+ \EX\left[ \CHT_{B_{n}} \left(\abs{f-n^{2}}-n^{2}\right)\right] \\
		&\qquad + \EX\left[ \CHT_{\Omega \setminus (A_{n}\cup B_{n})} \abs{f}\right] \\
		&\underset{n\to\infty}\longrightarrow 
			\EX\left[\abs{f}\right] = \mathbf{h}_{0}(f).
\end{align*}
This shows that although the sequence $(g_{n})_{n\in\N}$ does not converge to zero in 
$L_{1}$-norm, it does converge to zero in the metric compactification $\overline{L_{1}}^{h}$. 
The internal metric functional $\mathbf{h}_{0}$ has the representation \eqref{MetricFuncL1}
with the (constant) random measure $\omega\mapsto\xi_{\omega}=\delta_{\eta_{0}}$ with 
$\eta_{0}\in\overline{\R}^{h}$.
\end{example}

\begin{example}
For each $n\in\N$, let $A_{n}$ and $B_{n}$ be the intervals defined in the previous example. 
Now, we define $g_{n}\in L_{1}$ by $g_{n}:= -n\CHT_{A_{n}} + n\CHT_{B_{n}}$.
Then
\begin{equation*}
 	\norm{g_{n}}_{L_{1}}=2n/(n+1) < 2 \text{ for all } n \in \N. 
\end{equation*}
The sequence $(g_{n})_{n\in\N}$ is 
bounded in $L_{1}$, but it does not converge to zero in $L_{1}$-norm. However, as in the 
previous example, we have $\mathbf{h}_{g_{n}}\to \mathbf{h}_{0}$ as $n\to\infty$.
\end{example}

\begin{example}
Let $A$ be a measurable subset of $\Omega$ with $\PR(A)>0$. Suppose that $g:\Omega\to\R$ is 
an element of $L_{1}$. For each $n \in \N$, we define $g_{n} \in L_{1}$ by
$g_{n}:= n\CHT_{A} + g\CHT_{\Omega\setminus A}$ .
It follows that
\begin{equation*}
	\norm{g_{n}}_{L_{1}} = n\PR(A) + \EX\left[\CHT_{\Omega\setminus A}\abs{g}\right] 
		\to\infty \text{ as } n\to\infty.
\end{equation*}
On the other hand, for every $f\in L_{1}$ we obtain 
\begin{align*}
	\mathbf{h}_{g_{n}}(f) 
	    &= \EX\left[\abs{f-g_{n}}-\abs{g_{n}}\right] \\
	    &= \EX\left[\CHT_{A}\left(\abs{f-n}-n\right)\right]
	    		+ \EX\left[\CHT_{\Omega\setminus A}\left(\abs{f-g}-\abs{g}\right)\right] \\
		& \underset{n\to\infty}\longrightarrow
			\EX\left[-\CHT_{A}f\right] 
			+ \EX\left[\CHT_{\Omega\setminus A}\left(\abs{f-g}-\abs{g}\right)\right] = \mathbf{h}(f).					
\end{align*}
Hence $\mathbf{h}_{g_{n}}$ converges to the metric functional $\mathbf{h}$ which 
has the representation formula \eqref{MetricFuncL1}, where $\xi$ is the random measure on $\overline{\R}^{h}$
given by 
\begin{equation*}
	\xi = \CHT_{A}\delta_{\eta_{+\infty}}+\CHT_{\Omega\setminus A}\delta_{\eta_{g}}.
\end{equation*}
\end{example}

\begin{example}
Let $\Omega$ be the interval $\left[0,1\right]$ and let $\PR$ be the Lebesgue measure. Consider 
the Rademacher sequence $(g_{n})_{n\in\N}$ defined by $g_{n}(\omega):=\SGN(\sin(2^{n}\pi\omega))$
for all $n\in\N$ and $\omega\in\Omega$. It is not difficult to verify that $\mathbf{h}_{g_{n}}$ 
converges to the metric functional $\mathbf{h}$ with the representation formula \eqref{MetricFuncL1},
where $\xi$ is the constant random measure 
$\omega\mapsto\xi_{\omega}=\dfrac{1}{2}\delta_{\eta_{-1}}+\dfrac{1}{2}\delta_{\eta_{+1}}$.
\end{example}

The above examples are merely to demonstrate some of the different cases that can appear when 
determining the metric functionals on $L_{1}$. In \Cref{applications} we present some applications 
of the metric compactification of $L_{p}$ spaces. 

\begin{theorem}\label{Lpcompactification}
	Let $p \in ]1,\infty[$. 
	Every element of $\overline{L_{p}}^{h}$ has exactly one of the following forms:
	either
	\begin{equation}\label{FiniteMFLp}
		f\mapsto\mathbf{h}(f)=\left(\EX\left[
			\int_{\R}\abs{f-r}^{p}d\xi(r)\right]
			- \EX\left[\int_{\R}\abs{r}^{p}d\xi(r)\right]
			+ c^{p}\right)^{1/p} - c,
	\end{equation}
	where $\xi$ is a random measure on $\R$ and 	
	$c^{p}\geq \EX\left[\int_{\R}\abs{r}^{p}d\xi(r)\right]$;
	or  
	\begin{equation}\label{InfiniteMFLp}
			f\mapsto\mathbf{h}(f)=-\EX\left[f\zeta\right],
	\end{equation} 
	where $\zeta$ is an element of the closed unit ball of $L_{p/(p-1)}$.
\end{theorem}

\begin{remark}
	If the metric functional \eqref{FiniteMFLp} is represented by the random measure $\xi$ 
	on $\R$ given by $\omega\mapsto\xi_{\omega}=\delta_{g(\omega)}$ with $g\in L_{p}$, then 
	the representation formula \eqref{FiniteMFLp} becomes
	\begin{align}
		f \mapsto \mathbf{h}(f) 
			&= \big(\EX\left[\abs{f-g}^{p}\right] - \EX\left[\abs{g}^{p}\right] 
				+ c^{p}\big)^{1/p} - c \nonumber \\ 
			&= \big(\norm{f-g}_{L_{p}}^{p}-\norm{g}_{L_{p}}^{p} 
				+ c^{p} \big)^{1/p} - c \;\text{ with } c^{p} \geq \norm{g}_{L_{p}}^{p}. \label{lpMF}
	\end{align}
\end{remark}

In \cite{Gutierrez2018}, the author showed that bounded nets in the infinite-dimensional 
$\ell_{p}$ space with $p \in ]1,\infty[$ can only produce metric functionals represented by formulas
analogous to \eqref{lpMF}. In the present paper, bounded nets in non-atomic 
$L_p$ spaces with $p \in ]1,\infty[$ yield metric functionals represented by the general formulas 
\eqref{FiniteMFLp}. 

\section{Proofs}
Throughout this section we use well-known facts about representations of certain dual spaces.
These results can be found in standard functional analysis books such as 
\cite{Dunford_Schwartz1958} or \cite{Edwards1995}.
Recall that $\Lambda_{\infty}(\Omega,\Meas(\overline{\R}^{h}))$ is the space of mappings 
$\omega\mapsto\xi_{\omega}$ from $\Omega$ to $\Meas(\overline{\R}^{h})$ such that for each 
$\varphi\in C(\overline{\R}^{h})$ the real-valued function
\begin{equation*}
	\omega\mapsto\iprod{\varphi,\xi_{\omega}}
		=\int_{\overline{\R}^{h}}\varphi(\eta) d\xi_{\omega}(\eta)
\end{equation*} 
is measurable, and the function $\omega\mapsto\norm{\xi_{\omega}}_{\Meas(\overline{\R}^{h})}$ 
is essentially bounded. The linear space $\Lambda_{\infty}(\Omega,\Meas(\overline{\R}^{h}))$ 
is a Banach space with respect to the norm 
\begin{equation*}
	\norm{\xi}_{\Lambda_{\infty}} = \esssup_{\omega\in\Omega}\norm{\xi_{\omega}}_{\Meas(\overline{\R}^{h})}. 
\end{equation*}
Let $L_{1}(\Omega,C(\overline{\R}^{h}))$ denote the linear space of all mappings 
$\psi:\Omega\to C(\overline{\R}^{h})$ such that the real-valued function 
$\omega\mapsto \norm{\psi_{\omega}}_{C(\overline{\R}^{h})}$ is measurable and 
$\EX\left[\norm{\psi}_{C(\overline{\R}^{h})}\right]$ is finite. The real-valued function
$\psi \mapsto \EX\left[\norm{\psi}_{C(\overline{\R}^{h})}\right]$ defines a norm on 
$L_{1}(\Omega,C(\overline{\R}^{h}))$.
The dual space $L_{1}(\Omega,C(\overline{\R}^{h}))^{*}$, equipped with the usual dual norm,
has the representation
\begin{equation}\label{VectorValuedL1duality}
	L_{1}(\Omega,C(\overline{\R}^{h}))^{*}\cong
	\Lambda_{\infty}(\Omega,\Meas(\overline{\R}^{h}))
\end{equation}
under the isometric isomorphism 
$\vartheta :\Lambda_{\infty}(\Omega,\Meas(\overline{\R}^{h}))\to L_{1}(\Omega,C(\overline{\R}^{h}))^{*}$, 
$\xi \mapsto \vartheta(\xi)$ defined by 
\begin{equation*}
	[\vartheta(\xi)](\psi) 
		:= \EX\left[\int_{\overline{\R}^{h}}\psi(\eta)d\xi(\eta)\right] 
				\text{ for all } \psi\in L_{1}(\Omega,C(\overline{\R}^{h})).
\end{equation*}

\begin{lemma}\label{lemmaA}
	If $(g_{\alpha})_{\alpha}$ is a net in $L_{1}$, then there exists a random measure $\xi$ on 
	$\overline{\R}^{h}$ and a subnet $(g_{\beta})_{\beta}$ such that the net of random measures 
	$(\delta_{\eta_{g_{\beta}}})_{\beta}$ on $\overline{\R}^{h}$ converges weakly-star to $\xi$ 
	in $L_{1}(\Omega,C(\overline{\R}^{h}))^{*}$. That is, 
\begin{equation}\label{WeakStarLimitL1}
	\EX\left[\psi(\eta_{g_{\beta}})\right] 
		{\underset{\beta}\longrightarrow} 
		\EX\left[\int_{\overline{\R}^{h}}\psi(\eta)d\xi(\eta)\right]
		\text{ for all } \psi\in L_{1}(\Omega,C(\overline{\R}^{h})). 
\end{equation}
Moreover, if the net $(g_{\beta})_{\beta}$
is bounded in $L_{1}$, i.e., $\sup_{\beta}\norm{g_{\beta}}_{L_{1}}<\infty$, then $\xi$ in 
\eqref{WeakStarLimitL1} is a random measure on $\R$.
\end{lemma}

\begin{proof}
For each $\alpha$, the mapping $\omega\mapsto\delta_{\eta_{g_{\alpha}(\omega)}}$ defines a 
random measure on $\overline{\R}^{h}$, i.e., 
$\delta_{\eta_{g_{\alpha}}}\in\Lambda_{\infty}(\Omega,\Meas(\overline{\R}^{h}))$ and
$\delta_{\eta_{g_{\alpha}(\omega)}}\in\Prob(\overline{\R}^{h})$ for $\PR$-almost every $\omega\in\Omega$. 
By the representation \eqref{VectorValuedL1duality}, the net 
$(\delta_{\eta_{g_{\alpha}}})_{\alpha}$ lies in the closed unit ball of the dual space 
$L_{1}(\Omega,C(\overline{\R}^{h}))^{*}$. By the Banach-Alaouglu theorem, there exists 
a subnet $(\delta_{\eta_{g_{\beta}}})_{\beta}$ and 
an element $\xi \in \Lambda_{\infty}(\Omega,\Meas(\overline{\R}^{h}))$ with $\norm{\xi}_{\Lambda_{\infty}}\leq 1$ 
such that $\delta_{\eta_{g_{\beta}}}$ converges weakly-star to $\xi$ in 
$L_{1}(\Omega,C(\overline{\R}^{h}))^{*}$. Hence \eqref{WeakStarLimitL1} holds.

Next, we will show that $\xi$ is a random measure on $\overline{\R}^{h}$, i.e., 
$\xi_{\omega}\in\Prob(\overline{\R}^{h})$ for $\PR$-almost every $\omega\in\Omega$. 
Since $C(\overline{\R}^{h})$ is separable, the space $C(\overline{\R}^{h})_+$ of
non-negative continuous functions on $\overline{\R}^{h}$ contains a dense countable
subset, say $(\varphi_k)_{k\in\N}$.
Fix $k\in\N$. For every non-negative integrable function $u:\Omega\to\R$,
we apply \eqref{WeakStarLimitL1} with the mapping $\psi \in L_{1}(\Omega,C(\overline{\R}^{h}))$,
$\omega \mapsto u(\omega)\varphi_k$ to obtain 
\begin{equation*}
		\EX\left[u\int_{\overline{\R}^{h}}\varphi_k(\eta)d\xi(\eta)\right] 
			= \lim_{\beta}\EX\left[u\varphi_k(\eta_{g_{\beta}})\right] \geq 0.
\end{equation*}
Hence there is $N_k \in \Sigma$ with $\PR(N_k)=0$ such that 
$\int_{\overline{\R}^{h}}\varphi_k(\eta)d\xi_\omega(\eta) \geq 0$ for all $\omega\in\Omega\setminus N_k$.
By letting $N=\cup_{k\in\N}N_k$, it follows that $\PR(N)=0$ and
\begin{equation}\label{positive}
	\int_{\overline{\R}^{h}}\varphi_k(\eta)d\xi_\omega(\eta) \geq 0 
	\text{ for all } \omega\in\Omega\setminus N \text{ and all } k\in\N.
\end{equation}
Let $\varphi:\overline{\R}^{h} \to \R$ be a non-negative continuous function.
By density, there exists a sequence $(\varphi_{k_i})_{i\in\N}$ satisfying \eqref{positive}
and such that $\norm{\varphi_{k_i} - \varphi}_{C(\overline{\R}^{h})} \to 0 \text{ as } i\to\infty$.
Therefore, 
\begin{equation*}
	\int_{\overline{\R}^{h}}\varphi(\eta)d\xi_\omega(\eta) \geq 0 
	\text{ for all } \omega\in\Omega\setminus N.
\end{equation*}
This shows that $\xi_\omega$ is a positive measure for all $\omega\in\Omega\setminus N$ with $\PR(N)=0$.
Furthermore, if we apply \eqref{WeakStarLimitL1} with the mapping 
$\omega \mapsto \psi_\omega = \CHT_{\overline{\R}^{h}}$, we obtain 
\begin{equation*}
	\norm{\xi}_{\Lambda_{\infty}} 
		\geq \EX\left[\norm{\xi}_{\Meas(\overline{\R}^{h})} \right] 
	    = \EX\left[\int_{\overline{\R}^{h}}\psi(\eta)d\xi(\eta)\right] 
	    = \lim_{\beta}\EX\left[\psi(\eta_{g_{\beta}})\right] 
	    = \PR(\Omega)=1.
\end{equation*}
Hence $\norm{\xi}_{\Lambda_{\infty}} = 1$, and therefore $\xi$ is a random measure on 
$\overline{\R}^{h}$.

For the last part of the lemma we need to prove that $\xi_{\omega}(\{\eta_{+\infty},\eta_{-\infty}\})=0$ 
for $\PR$-almost every $\omega\in\Omega$. Let $C=\sup_{\beta}\norm{g_{\beta}}_{L_1}$ and let $\epsilon$ be a 
small positive real number. Define the compact interval
\begin{equation*}
	K_{\epsilon}:=\left[-(C+1)\epsilon^{-1}\,,\,(C+1)\epsilon^{-1}\right].
\end{equation*}
Let $\varphi^{\epsilon}$ denote a continuous function on $\R$ with compact support such that
$\varphi^{\epsilon}=1$ on $K_{\epsilon}$ and $0\leq \varphi^{\epsilon}\leq 1$ on $\R$.
For each $\omega\in\Omega$ define the mapping $\psi_{\omega}^{\epsilon}\in C(\overline{\R}^{h})$ 
by
\begin{equation*}
	\psi_{\omega}^{\epsilon}(\eta):=
		\begin{cases}
			\varphi^{\epsilon}(r), &\text{ if } \eta=\eta_{r},\\
			0, &\text{ if } \eta=\eta_{\pm \infty}.
		\end{cases}
\end{equation*}
Then, for every $\beta$ we have
\begin{align*}
	0  \leq \EX\left[1-\psi^{\epsilon}(\eta_{g_{\beta}})\right] 
	  & = \EX\left[\CHT_{\Omega\setminus g_{\beta}^{-1}(K_{\epsilon})}
			\left(1-\psi^{\epsilon}(\eta_{g_{\beta}})\right)\right] \\
	  & \leq \PR (\Omega\setminus g_{\beta}^{-1}(K_{\epsilon})) \\
	  & = \PR\big(\{\omega\in\Omega \mid \abs{g_{\beta}(\omega)} > (C+1)\epsilon^{-1} \}\big) \\
	  & \leq (C+1)^{-1}\epsilon\EX\left[\abs{g_{\beta}}\right] \\
	  & \leq (C+1)^{-1}\epsilon C < \epsilon.	
\end{align*}
Hence, by applying \eqref{WeakStarLimitL1} with $\psi=\psi^{\epsilon}$, we obtain 
\begin{equation*}
	0 \leq \EX\left[1-\int_{\overline{\R}^{h}}\psi^{\epsilon}(\eta)d\xi(\eta)\right] < \epsilon.
\end{equation*}
However, since $\R$ is identified in \eqref{Rcompactified} with the set $\{\eta_{r}\,\vert\,r\in\R\}$, 
it follows that 
\begin{equation*}
	\EX\left[1-\int_{\overline{\R}^{h}}\psi^{\epsilon}(\eta)d\xi(\eta)\right]
		= \EX\left[1-\int_{\R}\varphi^{\epsilon}(r)d\xi(r)\right] 
	\underset{\epsilon\to 0}\longrightarrow\EX\left[1-\xi(\R)\right].
\end{equation*}
Therefore $\xi_{\omega}(\R)=1$ for $\PR$-almost every $\omega\in\Omega$. This completes the proof of 
the lemma.
\end{proof}
\begin{proof}[\bfseries Proof of \Cref{L1compactification}]
If $\mathbf{h}$ is an element of $\overline{L_{1}}^{h}$, then there exists a net 
$(g_{\alpha})_{\alpha}$ in $L_{1}$ such that $\mathbf{h}_{g_{\alpha}}$ converges pointwise 
on $L_{1}$ to $\mathbf{h}$. By \Cref{lemmaA}, there exists a subnet $(g_{\beta})_{\beta}$ 
and a random measure $\xi$ on $\overline{\R}^{h}$ such that the net of random measures 
$(\delta_{\eta_{g_{\beta}}})_{\beta}$ on $\overline{\R}^{h}$ converges weakly-star to 
$\xi$ in $L_{1}(\Omega,C(\overline{\R}^{h}))^{*}$. 

Let $f$ be an element of $L_{1}$. For each $\beta$, the internal metric functional 
$\mathbf{h}_{g_{\beta}}$ on $L_{1}$ can be written as
\begin{equation}\label{auxIMFL1}
	\mathbf{h}_{g_{\beta}}(f)
		= \EX\left[\abs{f-g_{\beta}}-\abs{g_{\beta}}\right]
		= \EX\left[\int_{\overline{\R}^{h}}
			\eta(f)d\delta_{\eta_{g_{\beta}}}(\eta)\right]
		= \EX\left[\eta_{g_{\beta}}(f)\right].
\end{equation}
Consider the mapping $\omega\mapsto\psi_{\omega}^{f}$ from $\Omega$ to $C(\overline{\R}^{h})$ 
defined by $\psi_{\omega}^{f}(\eta):=\eta(f(\omega))$ for all $\eta\in\overline{\R}^{h}$. 
We claim that $\psi^{f}\in L_{1}(\Omega,C(\overline{\R}^{h}))$. Indeed, if $(\eta^{(k)})_{k\in\N}$ 
is a sequence in $\overline{\R}^{h}$ converging to some $\eta\in\overline{\R}^{h}$, then for 
$\PR$-almost every $\omega\in\Omega$ we have $\eta^{(k)}(f(\omega))\to\eta(f(\omega))$ as $k\to\infty$. 
Moreover, the function $\omega\mapsto \norm{\psi_{\omega}^{f}}_{C(\overline{\R}^{h})}$ is 
measurable and 
\begin{equation*}
	\EX\left[\norm{\psi^{f}}_{C(\overline{\R}^{h})}\right]
		= \EX\left[\sup_{\eta\in\overline{\R}^{h}}\abs{\eta(f)}\right]\\
		= \int_{\Omega}\sup_{\eta\in\overline{\R}^{h}}\abs{\eta(f(\omega))}d\PR(\omega)\\
		\leq \EX\left[\abs{f}\right] < \infty.
\end{equation*}
Note now that the formula \eqref{auxIMFL1} becomes 
$\mathbf{h}_{g_{\beta}}(f)=\EX\left[\psi^{f}(\eta_{g_{\beta}})\right]$. By applying 
the limit \eqref{WeakStarLimitL1}, we obtain
\begin{equation*}
	\mathbf{h}(f) 
		= \lim_{\beta} \mathbf{h}_{g_{\beta}}(f) 
		= \EX\left[\int_{\overline{\R}^{h}}
			\eta(f)d\xi(\eta)\right].
\end{equation*}
	
Conversely, assume that $\xi$ is a random measure on $\overline{\R}^{h}$. We need to find a net 
$(g_{\gamma})_{\gamma}$ in $L_{1}$ such that $\mathbf{h}_{g_{\gamma}}$ converges pointwise on 
$L_{1}$ to the functional $\mathbf{h}$ given by the formula \eqref{MetricFuncL1}. For this 
purpose, we introduce first some notations: let $\gamma$ denote a finite measurable partition 
$\left[S_{\gamma}^{1},...,S_{\gamma}^{\abs{\gamma}}\right]$ of $\Omega$, i.e., 
$\Omega=\cup_{j=1}^{\abs{\gamma}}S_{\gamma}^{j}$ with $S_{\gamma}^{j}\in\Sigma$ and 
$\PR(S_{\gamma}^{j})>0$ for all $j=1,...,\abs{\gamma}$, and also $S_{\gamma}^{j}\cap S_{\gamma}^{k}=\emptyset$ 
for $j\neq k$. Here $\abs{\gamma}$ denotes the number of elements of $\gamma$. Denote by $\T$ 
the collection of all finite measurable partitions of $\Omega$. The set $\T$ becomes a directed 
set with the partial order $\succcurlyeq$ defined by $\gamma\succcurlyeq\tilde{\gamma}$ if and only 
if $\gamma$ is a refinement of $\tilde{\gamma}$, i.e., for each $S_{\gamma}^{j}\in\gamma$ there 
exists $S_{\tilde{\gamma}}^{k}\in\tilde{\gamma}$ such that 
$\PR(S_{\gamma}^{j}\setminus S_{\tilde{\gamma}}^{k})=0$.

Let us first suppose that the random measure $\xi$ is constant of the form
\begin{equation}\label{constantDiracsum}
	\xi_{\omega}=\sum_{k=1}^{N}\theta_{k}\delta_{\eta^{(k)}},
\end{equation}
where $N\in\N$, $\sum_{k=1}^{N}\theta_{k}=1$, $\theta_{k}\in\left[0,1\right]\cap\Q$ and 
$\eta^{(k)}\in\overline{\R}^{h}$ for all $k=1,...,N$. Since the probability measure $\PR$ is 
non-atomic, for each finite measurable partition $\gamma=\left[S_{\gamma}^{1},...,S_{\gamma}^{\abs{\gamma}}\right]$ 
we can divide each $S_{\gamma}^{j}$ into further $N$ pairwise disjoint subsets 
$\{S_{\gamma}^{j,1},...,S_{\gamma}^{j,N}\}$ such that 
$\PR(S_{\gamma}^{j,k})=\theta_{k}\PR(S_{\gamma}^{j})$ for all $k=1,...,N$. Now, we can define 
the net $(g_{\gamma})_{\gamma\in\T}$ in $L_{1}$ by
\begin{equation}\label{TheNet}
	g_{\gamma}(\omega):=\begin{cases}
						-\abs{\gamma}, &\text{if } \eta^{(k)}=\eta_{-\infty}, \\
						+\abs{\gamma}, &\text{if } \eta^{(k)}=\eta_{+\infty}, \\		
						r_{k}, &\text{if } \eta^{(k)}=\eta_{r_{k}},				
					\end{cases}
\end{equation} 
for all $\omega\in S_{\gamma}^{1,k} \cup\cdots\cup S_{\gamma}^{\abs{\gamma},k}$ with $k=1,...,N$. 
Then the net of internal metric functionals $(\mathbf{h}_{g_{\gamma}})_{\gamma\in\T}$ 
converges pointwise on $L_{1}$ to the functional $\mathbf{h}$ given by the formula \eqref{MetricFuncL1} 
and represented by the random measure \eqref{constantDiracsum}. 

Next, suppose that $\xi$ is a general random measure on $\overline{\R}^{h}$. For each finite 
measurable partition $\gamma=\left[S_{\gamma}^{1},...,S_{\gamma}^{\abs{\gamma}}\right]$ of 
$\Omega$ define $\xi^{(\gamma)}\in\Lambda_{\infty}(\Omega,\Meas(\overline{\R}^{h}))$ by
\begin{equation}\label{NetOfMesuresIndexedbyPartitions}
	\iprod{\varphi,\xi_{\omega}^{(\gamma)}}
		:=\sum_{j=1}^{\abs{\gamma}}\frac{1}{\PR(S_{\gamma}^{j})}
			\EX\left[\CHT_{S_{\gamma}^{j}}
			\int_{\overline{\R}^{h}}\varphi(\eta) d\xi(\eta)\right]
			\CHT_{S_{\gamma}^{j}}(\omega),
\end{equation}
for all $\varphi\in C(\overline{\R}^{h})$. It follows that 
$\xi_{\omega}^{(\gamma)}\in\Prob(\overline{\R}^{h})$ for $\PR$-almost every $\omega\in\Omega$, and hence 
$\xi^{(\gamma)}$ is a random measure on $\overline{\R}^{h}$. We proceed to prove that 
\begin{equation}\label{WeakStarApprox}
	\xi^{(\gamma)}\overset{w*}{\underset{\gamma}\longrightarrow}\xi 
		\text{ in } L_{1}(\Omega,C(\overline{\R}^{h}))^{*}. 
\end{equation}
Let $\psi$ be an element of $L_{1}(\Omega,C(\overline{\R}^{h}))$ and let $\epsilon$ be a 
small positive real number. Then there exists a measurable finite partition
$\gamma_{\epsilon}=\left[S_{\gamma_{\epsilon}}^{1},...,S_{\gamma_{\epsilon}}^{\abs{\gamma_{\epsilon}}}\right]$ 
of $\Omega$ and a $C(\overline{\R}^{h})$-valued simple function
\begin{equation*}
	\omega \mapsto \psi^{(\gamma_{\epsilon})}_{\omega} 
			= \sum_{j=1}^{\abs{\gamma_{\epsilon}}}\varphi_{j}\CHT_{S_{\gamma_{\epsilon}}^{j}}(\omega)
			\text{ with } \varphi_{j}\in C(\overline{\R}^{h}) \text{ for } j=1,...,\abs{\gamma_{\epsilon}} 
\end{equation*} 
such that $\EX\left[\norm{\psi^{(\gamma_{\epsilon})}-\psi}_{C(\overline{\R}^{h})}\right] < \epsilon/2$. 
Therefore, for every $\gamma\succcurlyeq\gamma_{\epsilon}$ we have
\begin{align*}
	\bigg\vert\EX\left[\int_{\overline{\R}^{h}}\psi(\eta)d\xi^{(\gamma)}(\eta)\right]
		&-\EX\left[\int_{\overline{\R}^{h}}\psi(\eta)d\xi(\eta)\right]\bigg\vert \\
		& \leq \EX\left[\int_{\overline{\R}^{h}}\abs{\psi(\eta)-\psi^{(\gamma_{\epsilon})}(\eta)}d\xi^{(\gamma)}(\eta)\right] \\
		&\quad + \abs{\EX\left[\int_{\overline{\R}^{h}}\psi^{(\gamma_{\epsilon})}(\eta)d\xi^{(\gamma)}(\eta)
			- \int_{\overline{\R}^{h}}\psi^{(\gamma_{\epsilon})}(\eta)d\xi(\eta)\right]} \\
		&\quad + \EX\left[\int_{\overline{\R}^{h}}\abs{\psi^{(\gamma_{\epsilon})}(\eta)-\psi(\eta)}d\xi(\eta)\right] \\
		&< \epsilon/2 + 0 + \epsilon/2 =\epsilon.
\end{align*}
Hence \eqref{WeakStarApprox} holds.
Finally, we observe that the random measure $\xi^{(\gamma)}$ given by \eqref{NetOfMesuresIndexedbyPartitions} 
is constant on each $S_{\gamma}^{j}$ of the finite partition 
$\gamma=\left[S_{\gamma}^{1},...,S_{\gamma}^{\abs{\gamma}}\right]$
of $\Omega$. More precisely, 
\begin{equation}\label{SimpleMeasures}
	\xi_{|S_{\gamma}^{j}}^{(\gamma)} \in \Prob(\overline{\R}^{h}) \text{ for each } j=1,...,\abs{\gamma}.
\end{equation}
Furthermore, each probability measure \eqref{SimpleMeasures} can be approximated by measures
of the form \eqref{constantDiracsum} with respect to the weak-star topology $\sigma(\Meas(\overline{\R}^{h}),C(\overline{\R}^{h}))$.
This permits us to construct a net of the form \eqref{TheNet} on each $S_{\gamma}^{j}$. By a simple diagonal argument
with respect to the directed set $\T$ of measurable finite partitions, we can construct a 
net $(g_{\gamma})_{\gamma\in\T}$ in $L_{1}$ such that $\mathbf{h}_{g_{\gamma}}$ converges pointwise on $L_{1}$,
as the partition $\gamma$ gets finer and finer, to the functional $\mathbf{h}$ given by the formula \eqref{MetricFuncL1}. 
\end{proof}

\begin{proof}[\bfseries Proof of \Cref{Lpcompactification}]
If $\mathbf{h}\in \overline{L_{p}}^{h}$ then there exists a net $(g_{\alpha})_{\alpha}$ in 
$L_{p}$ such that $\mathbf{h}_{g_{\alpha}}$ converges pointwise on $L_{p}$ to $\mathbf{h}$.
The net $(g_{\alpha})_{\alpha}$ is either bounded or unbounded with respect to the $L_{p}$-norm. 

Let us first suppose that the net is bounded in $L_{p}$, i.e., $\sup_{\alpha}\norm{g_{\alpha}}_{L_{p}} < \infty$. 
By taking a subnet if necessary, we may assume that 
\begin{equation*}
	\norm{g_{\alpha}}_{L_{p}}\underset{\alpha}{\longrightarrow} c. 
\end{equation*}
Due to the inclusion $L_{p}\subset L_{1}$, the net $(g_{\alpha})_{\alpha}$ is bounded with respect
to the $L_{1}$-norm. By \Cref{lemmaA}, there exists a subnet $(g_{\beta})_{\beta}$ and a 
random measure $\xi$ on $\R$ for which the limit \eqref{WeakStarLimitL1} holds.
Now, let $f$ be an element of $L_{p}$. The internal metric functional $\mathbf{h}_{g_{\beta}}$
on $L_{p}$ given by \eqref{MetricFuncLp} becomes
\begin{equation}\label{newxLpMF}
	\mathbf{h}_{g_{\beta}}(f) 
		= \left(\EX\left[\int_{\R} (\abs{f-r}^{p} - \abs{r}^{p}) d\delta_{g_{\beta}}(r)\right] 
			  + \norm{g_{\beta}}_{L_{p}}^{p} \right)^{1/p} - \norm{g_{\beta}}_{L_{p}}.
\end{equation}
Let $\psi^{f}:\Omega\to C(\overline{\R}^{h})$ be the mapping defined by
\begin{equation*}
	\psi_{\omega}^{f}(\eta)
		:=\begin{cases}
				\abs{f(\omega)-r}^{p}-\abs{r}^{p}	&\text{ if } \eta=\eta_{r},\\
				-\SGN(f(\omega))\infty		&\text{ if } \eta=\eta_{+\infty},\\
				 \SGN(f(\omega))\infty		&\text{ if } \eta=\eta_{-\infty}.
	\end{cases}
\end{equation*}
Then \eqref{newxLpMF} becomes
$\mathbf{h}_{g_{\beta}}(f) = \big(\EX\left[\psi^{f}(\eta_{g_{\beta}})\right]
			+ \norm{g_{\beta}}_{L_{p}}^{p}\big)^{1/p} - \norm{g_{\beta}}_{L_{p}}$.
Finally, due to the limit \eqref{WeakStarLimitL1} we obtain
\begin{align*}
	\mathbf{h}_{g_{\beta}}(f) 
		&{\underset{\beta}\longrightarrow} 
		   \left(\EX\left[\int_{\R}\left(\abs{f-r}^{p}-\abs{r}^{p}\right) d\xi(r)\right] 
			+ c^{p}\right)^{1/p} - c \\
		&= \left(\EX\left[\int_{\R}\abs{f-r}^{p}d\xi(r)\right] 
				- \EX\left[\int_{\R}\abs{r}^{p}d\xi(r)\right]
				+ c^{p}\right)^{1/p} - c,
\end{align*}
with $c^{p} - \EX\left[\int_{\R}\abs{r}^{p}d\xi(r)\right]\geq 0$. 

On the other hand, if the net $(g_{\alpha})_{\alpha}$ is unbounded in $L_{p}$, by taking a subnet,
we may assume that 
\begin{equation*}
	\norm{g_{\alpha}}_{L_{p}}\underset{\alpha}{\longrightarrow}\infty.
\end{equation*}
By uniform convexity of the dual space $L_{p}^{*}\cong L_{p/(p-1)}$ and
\cite[Lemma 5.3]{Gutierrez2018}, there exists a subnet $(g_\beta)_\beta$ and an element $\zeta$ 
of the closed unit ball of $L_{p/(p-1)}$ such that
\begin{equation*}
	\mathbf{h}_{g_{\beta}}(f)\underset{\beta}{\longrightarrow} -\EX\left[f\zeta\right]  
	\text{ for all } f\in L_{p}.
\end{equation*}

Conversely, assume that $\zeta$ is an arbitrary element of the closed unit ball of $L_{p/(p-1)}$.
Pick an increasing sequence $\{A_{n}\}_{n\in\N}$ of measurable subsets of $\Omega$ with 
$A_{1}\neq\emptyset$ and $\cup_{n\in\N}A_{n}=\Omega$. Define now the sequence 
$(\zeta_{n})_{n\in\N}$ in $L_{p/(p-1)}$ by
\begin{equation*}
	\zeta_{n}
		:= \CHT_{A_{n}}\zeta
			+ \CHT_{\Omega\setminus A_{n}}\Big(\dfrac{1-\norm{\zeta}_{L_{p/(p-1)}}^{p/(p-1)}}{\PR(\Omega\setminus A_{n})}
				+ \abs{\zeta}^{p/(p-1)}\Big)^{(p-1)/p}.
\end{equation*}
Hence $\norm{\zeta_{n}}_{L_{p/(p-1)}}=1$ for all $n\in\N$. Furthermore, we observe that
$\zeta_{n}$ converges weakly to $\zeta$ in $L_{p/(p-1)}$. 
Due to the $L_{p/(p-1)}/L_{p}$--duality, for each $n\in\N$
there exists $\tilde{g}_{n}\in L_{p}$ with $\norm{\tilde{g_{n}}}_{L_{p}}=1$ such that 
$\EX\left[ \tilde{g}_{n}\zeta_{n}\right]=1$. By letting $g_{n}=n\tilde{g}_{n}$ for each $n\in\N$
and proceeding as in the proof of \cite[Lemma 5.3]{Gutierrez2018}, we can show that 
$\lim_{n\to\infty}\mathbf{h}_{g_{n}}(f)= -\EX\left[f\zeta\right]$ for all $f\in L_{p}$.
\end{proof}
\section{Applications}\label{applications}
\subsection{The \texorpdfstring{$L_{p}$}{Lp}-mean ergodic theorem}
Let $p \in ]1,\infty[$. Assume that $(\Omega,\PR)$ is a standard probability space. Let $T$ be a 
linear operator on $L_{p}=L_{p}(\Omega,\PR)$ such that $\norm{Tf}_{L_{p}}\leq\norm{f}_{L_{p}}$ 
for all $f\in L_{p}$. For an arbitrary element $g\in L_{p}$ define the mapping 
$F_{g}: L_{p}\to L_{p}$ by $F_{g}(f):=Tf + g$ for all $f\in L_{p}$. Hence $F_{g}$ defines a $1$-Lipschitz 
self-mapping of $L_p$. We observe that 
\begin{equation*}
	F_{g}^{n}(0)=\sum_{k=0}^{n-1}T^{k}g \;\text{ for all } n\geq 1.
\end{equation*}
Karlsson's \emph{metric spectral principle} \cite{Karlsson2001, Karlsson2018} asserts that there 
exists a metric functional 
$\mathbf{h}\in \overline{L_{p}}^{h}$ such that
\begin{equation}\label{CR-MeanErgodicThm}
	\lim_{n\to\infty}-\frac{1}{n}\mathbf{h}(F_{g}^{n}(0))=\tau,
\end{equation}
where the \emph{escape rate} $\tau:=\lim_{n\to\infty}n^{-1}\norm{F_{g}^{n}(0)}_{L_{p}}$
is well-defined due to subadditivity of the sequence $(\norm{F_{g}^{n}(0)}_{L_{p}})_{n\geq 1}$.

If $\tau=0$ then we obtain the trivial strong limit
\begin{equation*}
	\lim_{n\to\infty}\frac{1}{n}F_{g}^{n}(0) 
		= \lim_{n\to\infty}\frac{1}{n}\sum_{k=0}^{n-1}T^{k}g = 0.
\end{equation*}

Suppose now that $\tau>0$. Then the metric functional $\mathbf{h}$ in \eqref{CR-MeanErgodicThm} 
must be unbounded from below. Hence it is neither of the form \eqref{FiniteMFLp} nor
the zero functional in \eqref{InfiniteMFLp}. Therefore, there exists 
$\zeta\in L_{p/(p-1)}$ with $0<\norm{\zeta}_{L_{p/(p-1)}}\leq 1$ such that 
$\mathbf{h}(f)=-\EX\left[f\zeta\right]$ for all $f\in L_{p}$. We now proceed to show that 
$\zeta$ must be an element of the unit sphere of $L_{p/(p-1)}$. Indeed, for every $n\geq 1$ 
we have 
\begin{equation*}
	-\frac{1}{n}\mathbf{h}(F_{g}^{n}(0))
		= \frac{1}{n}\EX\left[F_{g}^{n}(0)\zeta\right]
		\leq \frac{1}{n}\norm{F_{g}^{n}(0)}_{L_{p}}\norm{\zeta}_{L_{p/(p-1)}}.
\end{equation*}
By \eqref{CR-MeanErgodicThm}, it follows that $\tau\leq \tau\norm{\zeta}_{L_{p/(p-1)}}$.
Hence $\norm{\zeta}_{L_{p/(p-1)}}=1$. 
On the other hand, by $L_{p}/L_{p/(p-1)}$--duality, there exists $g^{*}\in L_{p}$ with $\norm{g^{*}}_{L_{p}}=1$ 
such that $\EX\left[g^{*}\zeta\right]=1$. Next, we claim that
\begin{equation}\label{ErgLim}
	\lim_{n\to\infty}\frac{1}{n}F_{g}^{n}(0) 
		= \lim_{n\to\infty}\frac{1}{n}\sum_{k=0}^{n-1}T^{k}g 
		= \tau g^{*}.
\end{equation}
Indeed, for every $n\geq 1$ we have
\begin{align*}
	\frac{1}{n}\norm{F_{g}^{n}(0)}_{L_{p}} + \norm{\tau g^{*}}_{L_{p}} 
		\geq \norm{\frac{1}{n}F_{g}^{n}(0) + \tau g^{*}}_{L_{p}}  
		&\geq \EX\left[\frac{1}{n}F_{g}^{n}(0)\zeta\right]
			+ \EX\left[\tau g^{*}\zeta\right] \\ 
		&= -\frac{1}{n}\mathbf{h}(F_{g}^{n}(0)) + \tau.	
\end{align*}
Due to \eqref{CR-MeanErgodicThm} we obtain 
$\norm{\frac{1}{n}F_{g}^{n}(0) + \tau g^{*}}_{L_{p}}\to 2\tau$ as $n\to\infty$. 
Since $L_{p}$ is uniformly convex, it follows that 
$\norm{\frac{1}{n}F_{g}^{n}(0) - \tau g^{*}}_{L_{p}}\to 0$ as $n\to\infty$. 
Hence \eqref{ErgLim} holds.

\subsection{Alspach's fixed-point free isometry}
Alspach \cite{Alspach1981} presented an example of an isometry on a weakly compact convex 
subset of $L_{1}$ with no fixed points. More precisely, let $\Omega$ be the interval 
$[0,1]$ and let $\PR$ be the Lebesgue measure. Consider the subset $K$ of $L_{1}$ defined by
\begin{equation*}
	K=\left\{ f\in L_{1} \mid 
		0\leq f \leq 2 \text{ a.e. and } \EX\left[f\right] = 1 \right\}.
\end{equation*}
The set $K$ is a weakly compact convex subset of $L_{1}$. The mapping $F: K\to K$ defined by
\begin{equation*}
	F(f)(\omega):=\begin{cases}
		\min\{ 2\,,\,2f(2\omega) \}, &\text{if }\, 0\leq\omega\leq 1/2,\\
		\max\{ 0\,,\,2f(2\omega-1)-2 \}, &\text{if }\, 1/2<\omega \leq 1,
	\end{cases}
\end{equation*}
is an isometry, i.e., $\norm{F(f)-F(g)}_{L_{1}}=\norm{f-g}_{L_{1}}$	for all $f,g\in K$. 
Alspach proved that $F$ has no fixed points in $K$. We verify this fact by using metric 
functionals on $L_{1}$. 

First, we denote $g_{0}=\CHT_{\Omega}$ and $g_{n}=F(g_{n-1})$ for all
 $n\in\N$. More precisely, for every $n\in\N$ we have
\begin{equation*}
	g_{n}(\omega)=\begin{cases}
			2&\text{if }\, \omega\in \bigcup_{j=0}^{2^{n-1}-1} \left[2j2^{-n},(2j+1)2^{-n}\right],\\
			0&\text{otherwise}.
		\end{cases}
\end{equation*}
Equivalently, we can write $g_{n}=\CHT_{\Omega} + r_{n}$, where 
$r_{n}(\omega)=\SGN(\sin(2^{n}\pi\omega))$ is the $n$-th Rademacher function. Therefore, for 
every $f\in L_{1}$ we obtain
\begin{equation*}
		\lim_{n\to\infty}\mathbf{h}_{g_{n}}(f)
			= \mathbf{h}(f)
			= \EX\left[\int_{\overline{\R}^{h}}\eta(f)d\xi(\eta)\right],
\end{equation*}
where $\xi$ is the constant random measure on $\overline{\R}^{h}$ given by 
$\xi_{\omega}=\frac{1}{2}\delta_{\eta_{0}}+\frac{1}{2}\delta_{\eta_{2}}$ with 
$\eta_{0},\eta_{2}\in\overline{\R}^{h}$. That is,
\begin{equation}\label{AlspachLIM}
	\mathbf{h}(f) = \EX\left[\frac{1}{2}\abs{f} + \frac{1}{2}\left(\abs{f-2}-2\right)\right]
	\text{ for all } f\in L_{1}. 
\end{equation} 
In particular, we observe that the metric functional \eqref{AlspachLIM}
vanishes on $K$, i.e., $\mathbf{h}(f)=0$ for all $f\in K$.

Now, suppose that $F$ has a fixed point $g^{*}$ in $K$. Hence 
\begin{align*}
		0=\mathbf{h}(g^{*})&= \lim_{n\to\infty}\mathbf{h}_{g_{n}}(g^{*})\\
									&= \lim_{n\to\infty} \left( \norm{F^{n}(g^{*}) - F^{n}(\CHT_{\Omega})}_{L_{1}}-
										\norm{F^{n}(\CHT_{\Omega})}_{L_{1}} \right)\\
									&= \norm{g^{*}-\CHT_{\Omega}}_{L_{1}}-1.
\end{align*}
The only solutions on $K$ to $\norm{g^{*}-\CHT_{\Omega}}_{L_{1}}=1$ are of the form $g^{*}=2\CHT_{A}$,
where $A$ is a set of measure $1/2$. However, this is not possible because we would have
\begin{equation*}
		\mathbf{h}= \lim_{n\to\infty}\mathbf{h}_{F^{n}(g^{*})}
						= \lim_{n\to\infty}\mathbf{h}_{g^{*}}
						= \mathbf{h}_{g^{*}},
\end{equation*}
which is a contradiction. Therefore $F$ has no fixed points in $K$. 

\section*{Acknowledgments}
The author is very grateful to Prof. Anders Karlsson, Prof. Kalle Kyt\"ol\"a,  
and Prof. Olavi Nevanlinna for many valuable discussions and
suggestions. 
The author is also thankful to the anonymous referees for valuable suggestions that
improved the presentation of this paper.

This work was supported by the Academy of Finland, Grant No. 288318.

\end{document}